\documentclass[11pt]{article} 

\usepackage[scr=boondox]{mathalfa}
\usepackage{hyperref}
\usepackage[margin=1.0in]{geometry}
\usepackage{todonotes}
\usepackage{comment}

\usepackage[OT2,T1]{fontenc}

\usepackage{graphicx}
\usepackage{epstopdf}
\usepackage{amssymb,amsmath,amscd,amsthm}
\usepackage{graphicx,psfrag,epsfig,multirow}
\usepackage{cancel}
\usepackage{url}
\usepackage{caption}
\newtheorem{thm}{Theorem} 
\newtheorem{lemma}{Lemma}  

\newtheorem{rem}{Remark} \def\illustration #1 by #2 (#3) (#4){\leavevmode
  \vbox to #2{
   \hrule width #1 height 0pt depth 0pt
   \vfill
   \special{illustration #3 scaled #4}}} 
\begin{document}

\title{Arnold Tongues in Area-Preserving Maps} 
 \author{ Mark Levi\thanks{Research partially supported by DMS-1909200. Email: mxl48@psu.edu, jingzhou@psu.edu}  \and Jing Zhou 
     } 
\date{%
    Department of Mathematics, Penn State University, Pennsylvania, USA\\[2ex]%
    \today
}

\bibliographystyle{plain} 
\maketitle

\begin{abstract}
In the early 60's J. B. Keller and D. Levy discovered a fundamental property:  the instability tongues in Mathieu-type equations lose sharpness with the addition of higher-frequency harmonics in the Mathieu potentials.  20 years later V. Arnold rediscovered  a similar  phenomenon on sharpness of Arnold tongues in circle maps (and rediscovered the result of Keller and Levy).  In this paper we find a third class of objects where a similarly flavored behavior takes place: area-preserving maps of the cylinder.  Speaking loosely, we show that periodic orbits of  standard maps are extra fragile with respect to added drift (i.e. non-exactness) if the potential of the map is a trigonometric polynomial. That is, higher-frequency harmonics make periodic orbits more robust with respect to ``drift". The observation was motivated by the study of traveling waves in the discretized sine-Gordon equation which in turn models a wide variety of physical systems. 
\end{abstract}


\section{Introduction}
  Understanding invariant sets of area-preserving maps is one of the central problems of dynamics, and one of the most studied -  starting with Poincar\'e's geometrical theorem \cite{poincare} \cite{birkhoff} \cite{franks1988}, through KAM theory \cite{moser1962} \cite{arnold1988} \cite{arnold1989} and the Aubry-Mather theory \cite{mather1982} \cite{aubry1983} \cite{mather1991}. All the results require the exactness assumption. 

Much less is known about area-preserving maps which are non-exact, such as the maps $\varphi$  of the cylinder  $ {\mathbb S}  \times {\mathbb R}  $ with a ``drift": 
\begin{equation} 
	\int_{\varphi (\gamma) } v \;dx - \int_\gamma  {v \;dx} = \delta\not=0 ,  
	\label{eq:nonexact}
\end{equation} 
were $\gamma$ is a closed curve encircling the cylinder $ {\mathbb R}\,{\rm mod} 1 \times {\mathbb R}  $  once, Figure \ref{fig:cylinder}. 

\begin{figure}[thb]
 	\captionsetup{format=hang}  
	\center{  \includegraphics{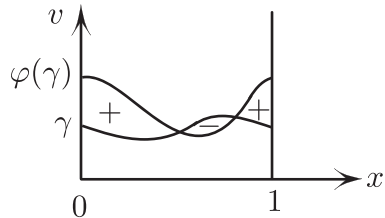}}
	\caption{Non-exact area-preserving cylinder map with $ \delta > 0 $. }  
	\label{fig:cylinder}
\end{figure}

Such maps are ubiquitous in Hamiltonian dynamics, arising in numerous settings.   We mention four examples. \\

 {\bf 1. The Frenkel-Kontorova model} of electrons in a chrystal lattice \cite{BraunKivsha} \cite{levi1990}.  
The model consists of an infinite chain of particles on the line a periodic potential $V(x)$ and with nearest neighbor coupling. The equilibria are the critical points  of the total energy 
\[
	\sum_{i\in {\mathbb Z}  } \frac{k}{2} (x_{i+1}-x_i) ^2 + V(x_i); 
\]  
although the sum is divergent, the variational equation, i.e. the equilibrium condition 
\begin{equation} 
	x_{i+1}-2x_i+x_{i-1} + k ^{-1}  V ^\prime (x_i )   
	\label{eq:FK}
\end{equation}  
is well defined. This  discrete analog of the Euler-Lagrange equation has a Hamiltonian counterpart obtained by the  introduction of  the discrete momentum  $v_i=x_i-x_{i-1}$: 
\begin{equation}  
   \left\{ \begin{array}{l} 
    x_{i+1} = x_i+v_i-V ^\prime (x_i)  \\[3pt] 
    v_{i+1} =v_i-V ^\prime (x_i)  \end{array}, \right.   
	\label{eq:standardmap}
\end{equation}  
 an area-preserving map.

If $ V ^\prime $ is periodic, then   (\ref{eq:standardmap})  defines a cylinder map. But a periodic $ V ^\prime $ leaves a possibility that  $V$ itself may have a ``tilt", i.e. a   linear part: 
\[
	V(x) = a \;x + V_{\rm periodic}(x), 
\] 
where   $V_{\rm periodic}(x+T)=V_{\rm periodic}(x)$ for some fixed $ T $.

The tilt causes the  map   (\ref{eq:standardmap})  to be non-exact:    (\ref{eq:nonexact})  holds for this map with $ \delta = -aT $. \\
\begin{figure}[thb]
 	\captionsetup{format=hang}  
	\center{  \includegraphics{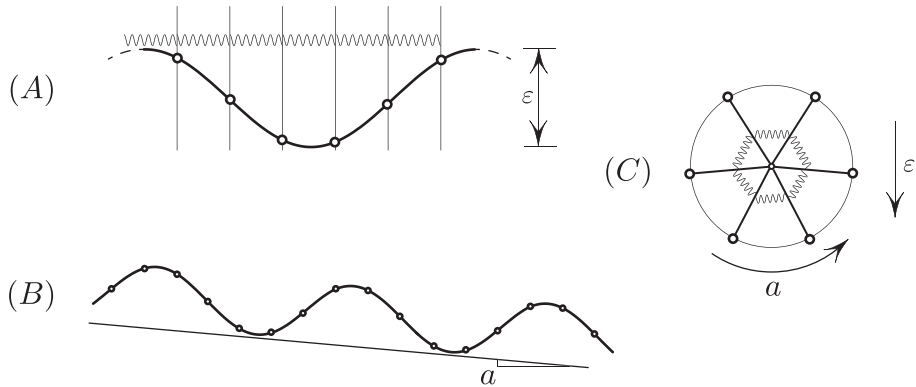}}
	\caption{(A): the  Frenkel-Kontorova model; (B): the tilt added, leading to the non-exact cylinder map; (C): tilt interpreted as torque acting on coupled pendula. }  
	\label{fig:FK}
\end{figure}

{\bf 2. Chains of coupled pendula.} In the special case of $ V(x) = c \sin x $ the Frenkel-Kontorova model has a mechanical interpretation as the chain of   torsionally coupled pendula (Figure \ref{fig:pendula});  here  $ x_i $ denotes the angle of the  $i$th pendulum with the downward vertical.  Now if each pendulum is subjected to a constant torque $ \tau $ then the potential acquires a linear part: $ V(x) = c \sin  x + \tau x $, and the corresponding cylinder map becomes non-exact, with $ \delta = - 2  \pi\tau $. 

\begin{figure}[thb]
 	\captionsetup{format=hang}  
	\center{  \includegraphics{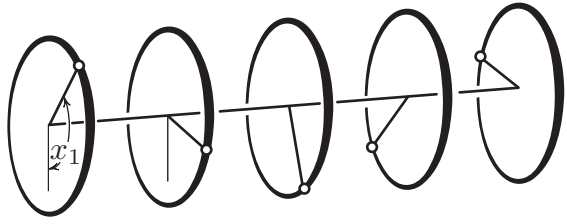}}
	\caption{Discretized sine-Gordon equation: pendula with nearest-neighbor torsional coupling.}  
	\label{fig:pendula}
\end{figure}

\vskip 0.1 in 

{\bf 3. Coupled Josephson junctions.} A Josephson junction consists of two superconductors separated by a narrow gap of a few angstroms \cite{JOSEPHSON1962}. The junction can behave as a superconductor or as a resistor, depending on the initial conditions (there were some hopes in the 1970s to use this property as a memory device). This behavior reminds of a pendulum with torque $ \delta $ described by
\begin{equation}
	\ddot x + c \dot x + \sin x = \delta  ; 
	\label{eq:pt}
\end{equation}
for the  $|\delta| <1$ there are two stable limiting regimes: either the stable equilibrium or the ``running" periodic solution corresponding to the tumbling motion: $x= \omega t+p(t)$ where $p$ is periodic. And in fact the same equation
(\ref{eq:pt})  is satisfied by the jump $ x =\arg \theta _2-\arg \theta _1 $ of the phase of the electron wave function across the gap if current $\delta$ is driven across the gap, Figure \ref{fig:josephson}.   The voltage across the gap is proportional to  the average $\langle\dot x \rangle$, or the average angular velocity in the pendulum interpretation.  
Thus the equilibrium solution with its average $\langle\dot x \rangle=0$ corresponds to zero voltage and thus to the superconducting state, while the tumbling solution with the voltage
$\langle\dot x \rangle\not=0$ corresponds to the resistive state. 
\begin{figure}[thb]
 	\captionsetup{format=hang}  
	\center{  \includegraphics{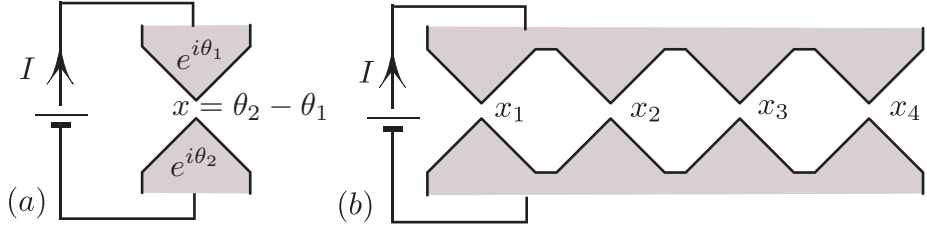}}
	\caption{Josephson junctions: single and coupled. Voltage across the junction is proportional to $ \langle \dot x \rangle  $.}  
	\label{fig:josephson}
\end{figure}
\vskip 0.1 in 
{\bf 4. Particle in $ {\mathbb R}  $  subject to a force that varies periodically both in time and position.} The motion of such a particle is governed by 
\[
	\ddot x = \Phi(x,t), 
\]  
where $ \Phi $ is periodic in both variables of period $1$ (without loss of generality). The Poincar\'e map $\varphi: (x, y)_{t=0}\mapsto (x, y)_{t=1} $, where $ y = \dot x $,   is generally non-exact, satisfying   (\ref{eq:nonexact})  with the drift equal to the average of force $\Phi$: 
\[
	\delta = \int_{0}^{1} \int_{0}^{1} \Phi(x,t)\;dx\;dt = \langle\Phi\rangle.
\]  
 \vskip 0.1 in 
 This completes our list of examples where non-exact area-preserving cylinder maps aries. \vskip 0.1 in 
 In this paper we show that periodic orbits of the standard map (\ref{eq:standardmap})  are extra sensitive to the added drift $\delta$ if the potential has harmonics of only low frequencies. There are (at least) two known phenomena with a similar flavor: (i) the sharpness of Arnold tongues in circle maps \cite{Arnold1983} 
\[x\mapsto x+ \omega + \varepsilon f(x) \]
where $f$ is a trigonometric polynomial is related to the degree of  $f$, and (ii) sharpness of  resonance zones in Hill's equations 
\[ \ddot x + (\omega ^2 + \varepsilon  q(t) ) x=0 \]
where $ q $ is a trigonometric polynomial 
is related to the degree of $ q $   \cite{LeKe}.  The present paper adds one more example to this list. According to Arnold \cite{Arnold1983}, Gelfand conjectured the existence of a general theorem which covers cases (i) and (ii); to this conjecture one can add the problem studied in the present paper. 

\vskip 10pt
\section{Results}
We consider periodic potentials with a linear part added: 

\[
   V(x) = \delta \;x + \varepsilon F(x), \ \ F(x+2\pi)=F(x),  
\]
so that the standard map (\ref{eq:standardmap}) with such $ V $ takes form
\begin{equation}  
   \left\{ \begin{array}{l} 
    x_{i+1} = x_i+v_i-\delta - \varepsilon f(x_i)  \\[3pt] 
    v_{i+1} =v_i-\delta - \varepsilon f(x_i)  \end{array} \right.   
	\label{eq:cylindermap}
\end{equation}  
where $f(x)=F'(x)$ is periodic of period $2\pi$.

For  $ \delta =0 $ the cylinder map (\ref{eq:cylindermap}) is exact, and it possesses a $p/q$ periodic orbit for any integer $ p, \ q\not=0 $, i.e. an orbit satisfying 
\[
   x_{i+q} = x_i + 2p\pi, \quad v_{i+q} = v_i;
\]  this  follows from Poincar\'e's Last Geometric Theorem as generalized by Franks \cite{franks1988}. In his generalization   Franks removed the requirement of  invariant boundary circles of an annulus, replacing it with   the assumption of exactness. Since the theorem no longer applies to the non-exact case  $ \delta \not=0$, a natural question arises: for what range of drift $ \delta $ do periodic orbits persist?    We show that if $V$ is a trigonometric polynomial  then this range becomes narrower if the degree of the trigonometric polynomial $f$  becomes smaller; and furthermore,  the tightness of the range is exponentially small in terms of the period $ q $.     More precisely, one has the following. 

\begin{thm}[Width of Arnold tongues]\label{thm:arnoldtongue} 
Let   $f(x)$  in  (\ref{eq:cylindermap}) be a trigonometric polynomial of degree $d$, and let $ p\geq 0 $, $ q>0 $ be integers. There exist positive constants $ \overline{\varepsilon}  $ and $ c  $ depending only on $q$ and $f$, such that for   any $0<\varepsilon<\overline{\varepsilon}$, all $ p/q $ periodic orbits of (\ref{eq:cylindermap}) disappear if the drift $|\delta |> c  \varepsilon^{[q/d]}$; here  $ [\cdot] $ denotes the integer part. 
\end{thm}
\begin{figure}[thb]
 	\captionsetup{format=hang}  
	\center{  \includegraphics{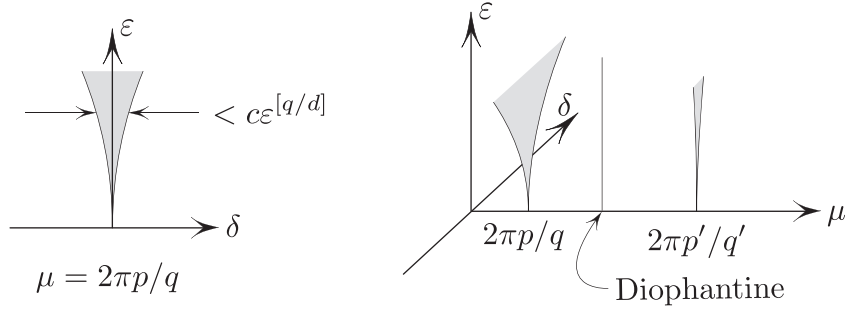}}
	\caption{Left: Arnold tongue for the Birkhoff periodic point with the rotation number $ \mu = p/q $ is exponentially narrow for trigonometric polynomials. Right: for Diophantine $ \mu $ one has an invariant KAM curve only for the drift $ \delta = 0 $, i.e. the tongue has zero width.}  
	\label{fig:arnoldtongue}
\end{figure} 
The other observation of this paper is that the $ p/q $ periodic points move on special curves as $ \delta $ changes.

\begin{thm}\label{thm:ycurve} 
Let $p>0$, $q>0$ be integers. If the perturbation term $f(x)$ in the cylinder map (\ref{eq:cylindermap}) is a $2\pi$-periodic analytic function ({\rm not necessarily a trigonometric polynomial}), then there exists a positive constant $\overline{\overline{\varepsilon}}$ depending only on $q$ and $f$, such that for any $|\varepsilon|<\overline{\overline{\varepsilon}}$ and $|\delta|<\overline{\overline{\varepsilon}}$, any $p/q$ periodic orbit of the perturbed map (\ref{eq:cylindermap}), if exists, lies on the graph of the function 
\begin{equation*}
    v=\mu + v_1(x)\varepsilon + v_2(x)\varepsilon^2+\cdots
\end{equation*}
where $ \mu = 2\pi p/q $ and  $v_n(x)$ is an $n^{\rm th}$ degree polynomial in $f(x+k\mu)$ ($0\le k\le q-1$) and their derivatives up to order $n-1$. In particular, 
\[
   v_1(x)=-\frac{q+1}{2}\overline{f}(x) + \overline{\overline{f}}(x),
\]
where 
\[
   \overline{f}(x) = \dfrac{1}{q} \sum_{k=0}^{q-1} f\left( x+k \mu  \right),
\]
\begin{equation*}
    {\overline{\overline f}}(x)=\dfrac{1}{q} \sum_{k=0}^{q-1} (q-k) f\left( x+k \mu  \right).
\end{equation*}

\end{thm}


\begin{rem}
   If $f$ is a trigonometric polynomial, then  $ {\overline f}$ and ${\overline {\overline f}} $ are  trigonometric polynomials as well, of degree not exceeding the degree $d$ of  $f$. Moreover, if $ d< q $ then $ \overline{f}=0 $. 
\end{rem}

 \begin{rem}  
 Birkhoff $ p/q $ periodic orbits  come in saddle-center pairs. As $ \delta  $ increases from $0$ to a critical value, a pair  disappears in a saddle-node bifurcation, Figure \ref{fig:channel}.  The first theorem therefore states that   critical values of $ \delta $  are $O(\varepsilon^{[q/d]})$. 
 \end{rem} 
 
 \begin{figure}[thb]
 	\captionsetup{format=hang}  
	\center{  \includegraphics{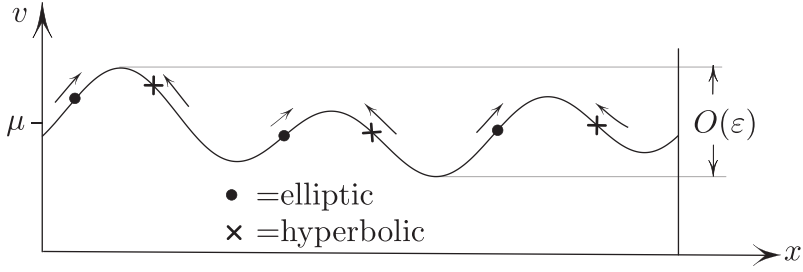}}
	\caption{Iterates of two periodic points (a center and a saddle). The arrows show the direction of motion of the iterates as $\delta$ increases. The two orbits disappear in a saddle-node bifurcation. As $ \delta $ changes, these points move with large speed (at least) $ O(\varepsilon^{-[q/d]}) $ (at bifurcation the speed becomes infinite).}  
	\label{fig:channel}
\end{figure} 

 \begin{rem}  
 A special case of Theorem \ref{thm:arnoldtongue} for $ q  \leq 3 $   was proven in \cite{saadatpour-levi} by a direct calculation. Unfortunately as $q$ increases, the complexity of this calculation tends to infinity. We overcome this problem by extending Arnold's approach \cite{Arnold1983} (that he used for circle maps) to the maps of the cylinder. 
 \end{rem} 
 
 \vskip 0.1 in 
\begin{rem} (An implication of Theorem \ref{thm:arnoldtongue} for traveling waves.) Consider an infinite periodic chain of pendula governed by the discretized sine-Gordon equation with damping: 
\begin{equation} 
	\ddot x_{k} + \gamma  \dot x_k+ \varepsilon  \sin x_k= (x_{k+1}-2x_k+x_{k-1}) + \delta.  
	\label{eq:sg}
\end{equation}   
Fixing $ q\in {\mathbb N}  $ and the ``twist" $p\in {\mathbb Z}  $, consider space-periodic ``twisted" solutions, i.e. the ones satisfying  
\begin{equation} 
	x_{k+q} (t)=x_k(t)  + 2 \pi p \  \  \hbox{for all}  \  \  t. 
	\label{eq:per}
\end{equation}   
According to \cite{saadatpour-levi} there exists a constant $ \varepsilon_0= \varepsilon_0( \gamma ) $ depending only on the damping coefficient $\gamma > 0 $ such that for all $ \varepsilon \in (0, \varepsilon_0)$ and for all $\delta$   every solution approaches asymptotically   either to an equilibrium or to a   traveling wave.  This wave satisfies
\[
 	x_k(t) = x_{k-1}(t+T/q) , 
\]  
(that is, each pendulum repeats what its neighbor is doing but with a delay); $q$-fold application of this identity gives 
\[
	x_k(t) = x_{k-q}(t+T) \buildrel{  (\ref{eq:per})}\over{=} x_k(t+T) - 2 \pi p, 
\]  
i.e. 
\[
	 x_k(t+T) = x_k(t) + 2 \pi p, 
\]
so that this solution is periodic modulo rotation $ x_k\mapsto x_k+ 2 \pi p $ ($k=1, \ldots, q) $ in $   {\mathbb S} = {\mathbb R}\mod 2\pi$.  Now according to Theorem \ref{thm:arnoldtongue} the equilibria of  (\ref{eq:sg})  disappear if $ \delta > c_0 \varepsilon ^q $ (a threshhold exponentially small in the number of pendula), and thus {\it    (\ref{eq:sg})  has  a globally attracting periodic traveling wave  for all such $\delta$. This traveling wave appears as the result of a saddle-node bifurcation of the equilibria which exist for smaller $ \delta $. }

\end{rem}
\vskip 0.1 in

\vskip 0.1 in

\noindent {\bf A remarkable fragility}  of equilibria is illustrated in Figure \ref{fig:FK}(C) showing an equilibrium of  $ q=6 $  coupled pendula with the same torque   applied to each.  With the choice of ``gravity'' $ \varepsilon = \frac{1}{2} $ the pendula  sag by a comparable amount and one might expect that the equilibrium could withstand the torque of a comparable magnitude. Yet, the equilibrium (and hence the corresponding periodic orbit of the map) disappears for  $ \delta > 0.0027 $, about $ .05\% $ of the ``gravity" $\varepsilon$!  Sinusoidal potentials are thus quite special: they are remarkably bad at pinning down the equilibria. 

\vskip 0.1 in 
The effect of non-exactness on the dynamics of the map is of interest in itself; its understanding would also shed light on physical effects, such as the fragility of Frenkel-Kontorova equilibria in crystals to imposed voltages.


\section{The Preliminaries}
The cylinder map (\ref{eq:cylindermap}) with $ \varepsilon = \delta = 0 $ has an invariant circle $ v=2 \pi p/q \buildrel {\rm def} \over = \mu$   consisting of $ p/q$ periodic points. This suggests     introducing a shifted momentum
$$  y := v - \mu  $$ 
with which  the cylinder map (\ref{eq:cylindermap}) takes a new form
\begin{equation}\label{map:main}
    \begin{cases}
    x_{i+1} = x_i + y_i + \mu + g(x_i;\varepsilon,\delta)\\
    y_{i+1} = y_i + g(x_i;\varepsilon,\delta)
    \end{cases}
\end{equation}
where 
$$g(x;\varepsilon,\delta):=-\delta-\varepsilon f(x),$$
and $f(x)$ is a $2\pi$-periodic analytic function as in Theorem \ref{thm:ycurve}.  

\medskip

This map (\ref{map:main}) restricted to a neighborhood of the circle $ y=0 $ can be viewed as a perturbation of the shift by $\mu$ in the  $x$ direction; the $ n $th iterate will thus be a perturbation of the shift by $ n \mu $ in the $x$-direction.  This suggests writing the $n$-th iterate of (\ref{map:main})   in the form
\begin{equation}\label{remainder}
    \begin{cases}
    x_n = x_0 + n \mu + \, _nR(x_0,y_0,\varepsilon,\delta)\\
    y_n = y_0 + \, _nS(x_0,y_0,\varepsilon,\delta);
    \end{cases}
\end{equation}
 a quick calculation shows that   
\begin{equation}\label{nRnS}
\begin{split}
    _nR(x_0,y_0,\varepsilon,\delta) &= n\left(y_0 + g(x_0)\right) + (n-1)g(x_1) + \cdots + 2g(x_{n-2}) + g(x_{n-1})\\
     _nS(x_0,y_0,\varepsilon,\delta) &= g(x_0) + g(x_1) + \cdots + g(x_{n-2}) + g(x_{n-1})
\end{split}
\end{equation}
where $\{x_0,x_1,\cdots,x_{n-1}\}$ are the $ x$-coordinates of the iterates of $ (x_0,y_0) $ under the cylinder map (\ref{map:main}).\\

The cylinder map (\ref{map:main}) possesses a $p/q$ periodic orbit:  $ x_q=x_0+ \underbrace{2  \pi p}_{q \mu} $, $ y_q=y_0 $  if and only if for some $(x_0,y_0)$ the remainders of the $q$-th iterate vanish: 
\begin{equation}\label{eq:remainders}
    _qR(x_0,y_0,\varepsilon,\delta) = \; _qS(x_0,y_0,\varepsilon,\delta)=0.
\end{equation}

The overall plan of the proof of Theorem 1 (the main result) is as follows.  Vanishing of the remainders (\ref{eq:remainders})  defines $ y_0 $ and $ \delta $ as functions of $ x_0 $ and $ \delta $: $y_0=Y(x_0,\varepsilon),\; \delta=\Delta(x_0,\varepsilon)$;  to prove the narrowness of the Arnold tongue (as specified in Theorem 1) amounts to showing that the range  $ \Delta ( {\mathbb R}, \varepsilon) $ is   $O( \varepsilon^r)$--small, where   $ r=[q/d]$. The proof of this narrow range statement goes as follows. Expanding $ \Delta $ in powers of $ \varepsilon $ we consider the first   $ x_0 $--dependent term: $ \Delta_r(x_0) \varepsilon ^r$. Our goal is to show that $ r>[q/d]$. To that end we prove (i) that  $\Delta_r(x_0) $ is periodic of period $ \mu $, and  (ii) that $ \Delta_r(x_0)$  is a trigonometric polynomial of degree  $ rd $. But a non-constant trigonometric polynomial periodic of period $ \mu = 2  \pi p/q $ must have degree $ > q $, i.e. $ rd>q$, or $ r> [q/d]$ as desired. 



\section{Structure of the remainders}

In this section we examine the $n$-th iterate of the cylinder map (\ref{map:main}) for any $n\in\mathbb{N}$ and the associated remainders.

 For brevity, we write $g(x) = g(x;\varepsilon, \delta)$  and $g'(x) = \frac{\partial}{\partial x}g(x;\varepsilon,\delta)=- \varepsilon f ^\prime (x)$ (recall that $ g(x;\varepsilon,\delta):=-\delta-\varepsilon f(x) $), and also 
  $$g_k^{(r)} := g^{(r)}(x_0+k\mu), \quad g_k^{(0)}=g(x_0+k \mu)=g_k;$$
  $$f_k^{(r)} := f^{(r)}(x_0+k\mu), \quad f_k^{(0)}=f(x_0+k \mu)=f_k.$$

The following lemma gives the structure of the remainders for any iterate of the cylinder map (\ref{map:main}). 

\begin{lemma}\label{lemma:remainders}
The remainders   $_nR$ and  $_nS$ are convergent series:  
$$ _nR=\; _nR_1 + \; _nR_2 + \cdots $$
$$ _nS=\; _nS_1 + \; _nS_2 + \cdots $$
where  $_nR_m$ and  $_nS_m$  are   homogeneous polynomials of degree $ m $ in terms of the items in the list  $\{y_0+g_0,\; g^{(l)}_k \; (0\le k\le n-1,\; 0\le l \le m-1)\}$. The two series converge for all $ x_0,y_0,\varepsilon,\delta$. Moreover, each term of degree $ m\geq 2$  contains at least one derivative of $g$ {\rm (so that, with $ g= - \delta - \varepsilon f(x) $ these terms are $ O( \varepsilon ) $)}.
 In particular, 
\begin{equation}\label{eq:firstorder}
\begin{split}
    _nR_1&= n(y_0+ g_0) + (n-1)g_1 + (n-2)g_2 + \cdots + g_{n-1},\\
    _nS_1&= g_0 + g_1 + \cdots + g_{n-1}.\\
\end{split}
\end{equation}
\end{lemma}

\vskip 0.2 in 




 \noindent {\bf Proof}  goes by  induction on $n$.
The statement is trivially true for $n=1$, and  we show that if the remainders $_nR $ and $_nS $ are of the form  claimed in the Lemma then the same is true for $ n+1 $. Indeed, 

 \[
 x_{n+1}= x_n + y_n + \mu + g(x_n) \buildrel ({\ref{remainder}}) \over = x_0+(n+1) \mu + \underbrace{\,_nR+\,_nS + g(x_n)}_{_{n+1}R}; 
\]
 and, thanks to the inductive assumption, it only remains to show that $ g(x_n) $ is of the form claimed. 
  Taylor expansion   yields

\begin{equation}
	 g(x_n)=g(x_0 + n \mu + \, _nR) =  g_n + \sum_{k\ge1}\frac{1}{k!}g^{(k)}_n \cdot (_nR)^k = g_n+\sum_{k\ge1}\frac{1}{k!}g^{(k)}_n \cdot \left( \sum_{j\ge1}\,_nR_j\right)^k 
	\label{eq:gn}
\end{equation}  
Now   $m$th degree part of  (\ref{eq:gn}) is a linear combination (with constant coefficients) of products 
\[
    g_n^{(k)} \, _nR_{j_1}\cdots\, _nR_{j_r}
\]  
with $ j_1+\cdots + j_r=m-1 $, and is thus a homogeneous polynomial of degree $ m $  as claimed. Furthermore, every term with $ m\geq 2 $ comes from the series in (\ref{eq:gn}) and thus contains a derivative of $ g $. 

The claim about $_{n+1}S $ is proven similarly: 
\[
    y_{n+1}=y_n+g(x_n)\buildrel ({\ref{remainder}}) \over =y_0+\underbrace{_nS+g(x_n)}_{_{n+1}S};
\]  
the rest of the proof is identical to the one above. 
This completes the proof of Lemma \ref{lemma:remainders}.
$\diamondsuit$

\vskip 0.2 in 
\begin{rem}
As an illustration of the lemma,  a short computation gives an explicit form of the degree-2 terms: 
\begin{equation*}
    \begin{split}
      _nR_2&= (n-2)g'_1 \cdot (y_0+ g_0) + (n-3)g'_2 \cdot (2(y_0+ g_0)+g_1) + \cdots\\
         &\quad + g'_{n-1} \cdot ((n-1)(y_0+ g_0)+(n-2)g_1 + \cdots + g_{n-2})\\
         _nS_2&= g'_1 \cdot (y_0+ g_0) + g'_2 \cdot (2(y_0+ g_0)+g_1) + \cdots\\
         &\quad + g'_{n-1} \cdot ((n-1)(y_0+ g_0)+(n-2)g_1 + \cdots + g_{n-2}).
    \end{split}
\end{equation*}
Each term in $_nR_2$, $_nS_2$ contains the first derivative of $g$ at some shift $x_0+k\mu$.
\end{rem}

\vskip 0.5in


\vskip 0.  in \section{The Existence of Periodic Orbits}
In this section we discuss the existence of the $p/q$ periodic orbits of the cylinder map (\ref{eq:cylindermap}) using Implicit Function Theorem and then we prove Theorem \ref{thm:ycurve}.\\

We begin by showing that   the equations 
$$_qR(x ,y , \varepsilon, \delta )=\, _qS(x ,y, \varepsilon, \delta )=0$$ 
uniquely determine $ \delta $ and $y$ as functions of $ x $ and $\varepsilon $ 
for all $ x\in {\mathbb R} $ and for all $ | \varepsilon | < \overline{\overline{\varepsilon}} $ for some positive 
$ \overline{\overline{\varepsilon}} $: 
\begin{equation}\label{eq:dy}
    \begin{cases}
       \delta = \Delta(x,\varepsilon)\\
       y = Y(x,\varepsilon) 
    \end{cases}
\end{equation}
such that $ _qR $ and $ _qS $ vanish identically if (\ref{eq:dy}) hold.

 Wishing to apply the implicit function theorem we note that $ _qR(x_0,0,0,0) = \; _qS(x_0,0,0,0)=0$ for all $x\in {\mathbb R} $
 and that, using Lemma \ref{lemma:remainders}
 \[
    \frac{ \partial ( _qR,\, _qS)}{ \partial ( \delta, y_0)}\biggl|_{y=\varepsilon=\delta=0} = \begin{pmatrix}
      -\dfrac{q(q+1)}{2} & q \\
      -q & 0
   \end{pmatrix}
\]  
 for all $ x $. Since the determinant  $ q ^2 \not=0 $,   the implicit function theorem applies: for any $ x_0$ there exists an open disk $ D_{x_0} $ centered at the
 point $ (x_0,  0 ) $ in the $ ( x, \varepsilon ) $-plane such that the implicit function $ (x, \varepsilon )\mapsto (\delta, y) $ 
 is well defined by the equations $ _qR=\,_qS=0 $ on the disk $ D_{x_0} $.   The segment 
 $ [0, 2  \pi ] \times \{0\} $ in the $ (x, \varepsilon ) $-plane is covered by open disks and thus has a finite subcover; but then this finite union contains a strip  nonzero width 
 $ \overline{\overline{\varepsilon}}> 0  $ around the $ x $--axis. Moreover, the functions corresponding to the overlapping disks   coincide. Thus the implicit function 
 $ (x, \varepsilon )\mapsto (\Delta, Y) $ is defined for all $ x $ and for all $ | \varepsilon | < \overline{\overline{\varepsilon }} $.

\begin{proof}[Proof of Theorem \ref{thm:ycurve}]
The proof proceeds by induction.  If $ (x_0, y_0) $ is a $ p/q $-periodic point of the map (\ref{eq:cylindermap}) for some $ \delta $, then (\ref{eq:dy})  holds; proving the theorem thus amounts to showing that the coefficients    in the expansion of $ Y(x_0, \varepsilon ) $ in powers of $ \varepsilon $  are polynomials in $ f_k $ and its derivatives up to order $ k-1 $. We fix $\varepsilon<\overline{\overline{\varepsilon}}$ so that $ \Delta $ and $ Y $ are well-defined.   
As the first step in induction we show that $\Delta_1$ and $ Y_1 $ in the expansions  
\begin{equation*}
    \begin{split}
        \Delta(x_0) &= \Delta_1(x_0)\varepsilon + o(\varepsilon),\\
        Y(x_0) &= Y_1(x_0)\varepsilon + o(\varepsilon)
    \end{split}
\end{equation*}
are polynomials of degree $1$ (i.e. linear) in $f_k$, where $0\le k\le q-1$. \\

Indeed,   by Lemma \ref{lemma:remainders} 
\[
   _qS(x_0,Y,\varepsilon,\Delta) = \sum_{m\ge1}\; _qS_m(x_0,Y,\varepsilon,\Delta),
\]
where by (\ref{eq:firstorder}) the linear term is 
\[
   _qS_1(x_0,Y,\varepsilon,\Delta) = -q\Delta-\varepsilon q\overline{f}(x_0)   \ \hbox{with} \  \overline{f}(x_0)= \frac{1}{q} \sum_{k=0}^{q-1} f_k,
\]
while each higher-degree term $_qS_m$ ($m\ge2$) is a   homogeneous polynomial of degree $m$  in the items from the list $\{Y-\Delta-\varepsilon f_0,\; -\Delta-\varepsilon f_k, \; \varepsilon f^{(l)}_k \} $ with $  0\le k\le q-1,\; 0\le l \le m-1$ and thus is of order $\mathcal{O}(\varepsilon^2)$ since   $\Delta,Y\sim\mathcal{O}(\varepsilon)$, so that  
\[
   _qS(x_0,Y,\varepsilon,\Delta) = -q\Delta-\varepsilon q\overline{f}(x_0) + o(\varepsilon).
\]
Since the above expression vanishes by the definition of $Y,\Delta$, we conclude that 
\begin{equation}
	 \Delta = - \varepsilon \overline{f}(x_0)+ o(\varepsilon),
	\label{eq:Delta1}
\end{equation}  
so that the leading coefficient of $ \varepsilon $ is 
\[
   \Delta_1(x_0) = -\overline{f}(x_0).
\] 
Similarly, by Lemma \ref{lemma:remainders}, we have 
\[
   _qR(x_0,Y,\varepsilon,\Delta) = \sum_{m\ge1}\; _qR_m(x_0,Y,\varepsilon,\Delta),
\]
where  
\[
   _qR_1(x_0,Y,\varepsilon,\Delta)  \buildrel{(\ref{eq:firstorder})}\over{=} qY-\dfrac{q(q+1)}{2}\Delta-\varepsilon \sum_{k=0}^{q-1} (q-k) f_k , 
\]
while each higher-degree term $_qR_m$ ($m\ge 2$) is a degree-$m$ homogeneous polynomial in the items from the list $\{Y-\Delta-\varepsilon f_0,\; -\Delta-\varepsilon f_k, \; \varepsilon f^{(l)}_k \}$, where $0\le k\le q-1,\; 0\le l \le m-1$; and since $Y,\Delta\sim\mathcal{O}(\varepsilon)$ all  $_qR_m$ with $ m\ge 2 $ are  at most  $\mathcal{O}(\varepsilon^2)$, so that 
\[
    \underbrace{_qR(x_0,Y,\varepsilon,\Delta)}_{=0} = qY-\dfrac{q(q+1)}{2}\Delta-\varepsilon \sum_{k=0}^{q-1} (q-k) f_k + o(\varepsilon).
\]
Substituting into this we obtain $  Y(x_0,\varepsilon) = Y_1(x_0)\varepsilon + o(\varepsilon)$ and (\ref{eq:Delta1}) results in   
\[
   Y_1(x_0) = -\frac{q+1}{2}\overline{f}(x_0) + \overline{\overline{f}}(x_0) ,
\]
where $\displaystyle \overline{\overline{f}}(x_0)=\dfrac{1}{q}\sum_{k=0}^{q-1} (q-k) f_k$.\\
This completes the first step of induction. To carry out the $n$th inductive step, let $ n>1 $ and assume that in the expansion
\begin{equation}\label{eq:case n}
    \begin{split}
        \Delta(x_0) &= \Delta_1(x_0)\varepsilon+\cdots+\Delta_n(x_0)\varepsilon^n + \ldots,\\
        Y(x_0)&=Y_1(x_0)\varepsilon+\cdots+Y_n(x_0)\varepsilon^n + \ldots
    \end{split}
\end{equation}
each  $\Delta_m(x_0)$ and $Y_m(x_0)$ with $ m\le n $  is a  polynomial of degree $m$ in $f_k^{(l)}$ with $0\le k\le q-1$, $0\le l \le m-1$. Our goal is  show that then the coefficients $\Delta_{n+1}$ and $Y_{n+1}$   are polynomials of degree $n+1 $ in $f_k^{(l)}$  with $0\le k\le q-1$, $0\le l \le n$.\\

Just like in the first step, we obtain   obtain $\Delta_{n+1}$ by extracting the coefficient of $ \varepsilon^{n+1}$ in  $_qS$: 
\[
   0 =\; _qS(x_0,Y,\varepsilon,\Delta) =\; _qS_1(x_0,Y,\varepsilon,\Delta) +  \sum_{m\ge2}^{n+1}\; _qS_m(x_0,Y,\varepsilon,\Delta) +  \sum_{m>n+1}\; _qS_m(x_0,Y,\varepsilon,\Delta). 
\]

For $m>n+1$, $qS_m(x_0,Y,\varepsilon,\Delta) \sim o(\varepsilon^{n+1})$, and hence the last sum does not contribute powers of degree $ n+1 $. On the other hand, since
\[
   _qS_1(x_0,Y,\varepsilon,\Delta) = -q\Delta-\varepsilon q\overline{f}(x_0), 
\]
 $_qS_1$ contributes  $-q\Delta_{n+1}(x_0)\varepsilon^{n+1}$, a {\it  constant} multiple of $ \Delta_{n+1}$. It thus suffices to show that  the  terms in the middle sum are polynomials of degree up to $ n+1 $ in terms of $ f^{(l)}_k $ with $ 0\le k\le q-1,\; 0\le l \le n $. 
 
For each $m\in[2,n+1]$, $_qS_m(x_0,Y,\varepsilon,\Delta)$ is a degree-$m$ homogeneous polynomial in the items from the list $\{Y-\Delta-\varepsilon f_0,\; -\Delta-\varepsilon f_k, \; \varepsilon f^{(l)}_k \} $ with $ 0\le k\le q-1,\; 0\le l \le m-1$. Thus by the inductive assumption, the coefficient of $ \varepsilon^{n+1} $  in $_qS_m$ $(2\le m\le n+1)$ is a linear combination of the  terms 
\[
   (Y_i\varepsilon^i)^{m_i}(\Delta_j\varepsilon^j)^{m_j}(f_k^{(l)}\varepsilon)^{m_s}
\]
with $i m_i + j m_j + m_s =n+1$, $0\le l\le m-1$.

Since by Lemma \ref{lemma:remainders}  each higher-degree term $_qS_m$ ($m\ge2$) has at least one derivative of $g$, it follows that $m_s\ge1$ and consequently $i,j\le n$, so that the inductive assumption applies to $ Y_i $ and $ \Delta_j $ above. And thus 
$Y_i$ is a  polynomial of degree $i$   in $f_k^{(l)}$ with $ 0\le k\le q-1 $ with $  0\le l\le i-1$, and similarly $\Delta_j $ is a polynomial of degree $j$  with $0\le k\le q-1,\; 0\le l\le j-1$. Therefore $(Y_i)^{m_i}(\Delta_j)^{m_j}(f_k^{(l)})^{m_s}$ is a polynomial of degree $n+1$  in $f_k^{(l)}$ with  $0\le k\le q-1,\; 0\le l\le n$. This completes the inductive step for $ \Delta_{n+1} $. The step for $ Y_{n+1} $ is carried out in the same way: we have \[
   0 =\; _qR(x_0,Y,\varepsilon,\Delta) =\; _qR_1(x_0,Y,\varepsilon,\Delta) +  \sum_{m\ge2}^{n+1}\; _qR_m(x_0,Y,\varepsilon,\Delta) +  \sum_{m>n+1}\; _qR_m(x_0,Y,\varepsilon,\Delta).
\]
Just as before,  the last sum does not contribute to the coefficient of $\varepsilon^{n+1} $. The first term 
\[
   _qR_1(x_0,Y,\varepsilon,\Delta) = qY-\dfrac{q(q+1)}{2}\Delta-\varepsilon q\overline{\overline{f}}(x_0),   
\]
  contributes 
\begin{equation}
	  \left(qY_{n+1}-\dfrac{q(q+1)}{2}\Delta_{n+1}\right)\varepsilon^{n+1}. 
	\label{eq:DY}
\end{equation}  
The coefficients of $ \varepsilon^{n+1} $ in the middle sum are polynomials of degree at most $n+1$ in terms of $f$, its shifts and its derivatives up to order $n$, precisely as we proved before when treating $\Delta_{n+1}$. This shows that the coefficient in  (\ref{eq:DY}) is a polynomial of degree at most $n+1$ in terms of $f$, its shifts and its derivatives up to order $n$. Since the same is true for $\Delta_{n+1} $, this holds for 
$Y_{n+1}$ as well, thus completing the induction step and the proof of Theorem \ref{thm:ycurve}.

\end{proof}

  \section{The Periodicity Lemma}
 
In this section we show that the leading $x${\it -dependent} coefficient in the $ \varepsilon $-expansion of $ \Delta (x, \varepsilon ) $    is periodic of period $ \mu$. This fact plays a key role in the proof of Theorem \ref{thm:arnoldtongue}. Before proving this periodicity we show that this leading coefficient is  also the leading term up to a constant factor in $_qR $ and $_qS $  as well, another crucial fact.

Let $r\ge1$ be the smallest power of  $\varepsilon$ where $x$ first appears in the coefficient of the   expansion of $ \Delta$ in powers of $\varepsilon $, so that
\begin{align}\label{eq:expdelta}
    \Delta(x_0,\varepsilon) &= A(\varepsilon) + \Delta_r(x_0)\varepsilon^{r} + o(\varepsilon^{r})
\end{align}
where $A(\varepsilon)$ is a polynomial in $\varepsilon$ of degree at most $r-1$ with constant coefficients. 
We claim that replacing $ \Delta  $ with its constant part $ A( \varepsilon ) $ in 
$ _qR(x_0,Y,\varepsilon,\Delta)=0 $ and $ _qS(x_0,Y,\varepsilon,\Delta) =0$ changes these from $ 0 $ by the amount proportional to $ \Delta_r(x_0) $ in the leading order: 
\begin{equation}\label{nRprincipal} _qR(x_0,Y,\varepsilon,A(\varepsilon)) =  \dfrac{q(q+1)}{2} \Delta_r(x_0) \varepsilon^r + o(\varepsilon^r) 
\end{equation}
and
\begin{equation}\label{nSprincipal} _qS(x_0,Y,\varepsilon,A(\varepsilon)) = q \Delta_r(x_0) \varepsilon^r + o(\varepsilon^r); 
\end{equation}
here $ Y=Y(x_0, \varepsilon ) $. 
Indeed, by Lemma \ref{lemma:remainders} we have  
\[
    _qR(x_0,Y,\varepsilon,\Delta) - \; _qR(x_0,Y,\varepsilon,A(\varepsilon)) 
    = \sum_{m\ge 1} \; _qR_m(x_0,Y,\varepsilon,\Delta) - \; _qR_m(x_0,Y,\varepsilon,A(\varepsilon)).
\]

Starting with $ m=1 $, the terms  $ _qR_1 $ and $ _qS_1 $  are linear in $ \delta $ with constant coefficients, according to (\ref{eq:firstorder}); more precisely,  
\begin{align*}
    &\quad_qR_1(x_0,Y,\varepsilon,\Delta) - \;  _qR_1(x_0,Y,\varepsilon,A(\varepsilon)) \\
    &  = q(Y-\Delta-\varepsilon f_0) + \sum_{k=1}^{q-1} (q-k) (-\Delta-\varepsilon f_k) \\
    &\quad - q(y_0-A(\varepsilon)-\varepsilon f_0) - \sum_{k=1}^{q-1} (q-k) (-A(\varepsilon)-\varepsilon f_k)\\
    &= -q\Delta_r(x_0)\varepsilon^r - \sum_{k=1}^{q-1} (q-k) \Delta_r(x_0)\varepsilon^r + o(\varepsilon^r)\\
    &= -\dfrac{q(q+1)}{2} \Delta_r(x_0) \varepsilon^r + o(\varepsilon^r).
\end{align*}
To complete the proof of (\ref{nRprincipal}) it suffices to show that      
\begin{equation}
	   _qR_m(x_0,Y,\varepsilon,\Delta) - \; _qR_m(x_0,Y,\varepsilon,A(\varepsilon)) = \mathcal{O}(\varepsilon^{r+1}) \ \ \hbox{for} \ \ m\ge 2.
	\label{eq:RR}
\end{equation}  
 
By Lemma \ref{lemma:remainders},  $_q R_m(x_0,Y,\varepsilon,\delta)$  is a  homogeneous polynomial of  degree $ m $ in the  items from the list $\{Y+g_0,g_1,\cdots,g_{q-1}, g^{(l)}_k\} $ with  $ 1\le l \le m-1, \; 0\le k\le q-1$ and it contains at least one derivative $g^{(l)}_k$ for some $1\le l \le m-1, \; 0\le k\le q-1$, thus contributing an extra factor of $\varepsilon$. Since $g(x_0;\varepsilon,\delta)=-\delta-\varepsilon f(x_0)$, replacing  $\delta =\Delta=A+\Delta_r \varepsilon^r +o(\varepsilon^{r})$ by $ A ( \varepsilon )  $ changes $_qR_m $ by   $\mathcal{O}(\varepsilon^r)\cdot \varepsilon= \mathcal{O}(\varepsilon^{r+1})$, thus completing the proof of (\ref{nRprincipal}). 
The proof of (\ref{nSprincipal}) is identical and therefore omitted.

\begin{lemma}[Periodicity]\label{lemma:periodicity}
For all sufficiently small $ \varepsilon $ the leading $x$-dependent coefficient  $\Delta_r$ in the expansion (\ref{eq:expdelta})  of $ \Delta $  is periodic in $\mu$: for any $x$ we have  $$\Delta_r(x+\mu)=\Delta_r(x).$$
\end{lemma}

\begin{proof}
We {\it  fix} an initial point $(x_0,y_0=Y(x_0,\varepsilon))$ and set $\delta=\Delta(x_0, \varepsilon )$ in the cylinder map (\ref{map:main}) (and consequently  $g(x_0;\varepsilon,\delta)=-\Delta(x_0) - \varepsilon f(x)$).  

For future use we observe (dropping $ \varepsilon $ from the notation for the sake of brevity) that 
\begin{equation}
	     Y(x_1 )=y_1, \ \ \Delta (x_1 ) = \Delta(x_0 )
	\label{eq:YD}
\end{equation}  
for all sufficiently small $ \varepsilon $. Indeed,    the orbit  $(x_0,y_0=Y(x_0)),(x_1,y_1),(x_2,y_2),\cdots$   is $q$-periodic under the map with $\delta=\Delta(x_0, \varepsilon )$ by the definition of $ Y $ and $ \Delta $; thus     $_qR $ and $ _qS $ vanish at any point of this orbit, and in particular at $ (x_1, y_1) $: 
\begin{equation}
	  _qR(x_1,y_1,\varepsilon,\Delta(x_0))=\; _qS(x_1,y_1,\varepsilon,\Delta(x_0))=0.
	\label{eq:RS1}
\end{equation} 
On the other hand, by the definition of $ Y$ and $ \Delta $ 
\begin{equation}
	  _qR(x_1,Y(x_1),\varepsilon,\Delta(x_1))=\; _qS(x_1,Y(x_1),\varepsilon,\Delta(x_1))=0; 
	\label{eq:RS2}
\end{equation} 
provided the conditions of the implicit function theorem in Section 5 are satisfied, the solution is unique, and thus comparison of (\ref{eq:RS1}) and (\ref{eq:RS2}) implies (\ref{eq:YD}). 
The conditions of the implicit function theorem are satisfied if $ \varepsilon $ is restricted to be sufficiently small - more precisely, so that $ |y_1 |< \overline{\overline{\varepsilon}} $. To that end we note that   
\[
    | y_1| =| Y(x_0)-\Delta(x_0)-\varepsilon f(x_0)| \leq c\varepsilon 
\]
where $c $ is a constant depending only on $ q $ and on $ \max |f| $. It thus suffices to restrict $ \varepsilon $ to $ \varepsilon < \bar{\varepsilon}:= \min (\overline{\overline{\varepsilon}}/c, \overline{\overline{\varepsilon}}) $, which we do from now on. 

We now proceed with the rest of the proof. 
Recalling that $ y_0=Y(x_0) $ we have
\begin{equation}
      _qS(x,Y(x),\varepsilon,A(\varepsilon)) \bigg|_{x=x_0} ^{x=x_1}\buildrel \ref{nSprincipal} \over {=} q (\Delta_r(x_1)-\Delta_r(x_0)) \varepsilon^r + o(\varepsilon^r). 
    \label{eq:BB}
\end{equation} 
  
And since  $x_1 = x_0+ \mu + y_0- \Delta (x_0) - \varepsilon f(x_0) = x_0+ \mu + O( \varepsilon ) $, this implies 
\begin{equation}
    _qS(x_1,Y(x_1),\varepsilon,A(\varepsilon)) -\, _qS(x_0,y_0,\varepsilon,A(\varepsilon)) = q(\Delta_r(x_0+ \mu )-\Delta_r(x_0)) \varepsilon^r + o( \varepsilon ^r). 
    \label{eq:dS}
\end{equation} 
We now show that the left-hand side is $o( \varepsilon^r) $ (thus completing the proof of the lemma).  Consider the orbit $ (\tilde x_k, \tilde y_k) $ of the same initial point $ (x_0, y_0) $ but under the map with $ \delta = A( \varepsilon ) $ (instead of $ \delta = \Delta (x_0) $).    We will   show that 
\begin{equation}
    _qS(\tilde x_1,\tilde y_1,\varepsilon,A(\varepsilon)) - \, _qS(x_0,y_0,\varepsilon,A(\varepsilon)) =   o( \varepsilon ^r) 
    \label{eq:first}
\end{equation} 
and 
\begin{equation}
    _qS(  x_1,Y(x_1),\varepsilon,A(\varepsilon)) - \, _qS(\tilde x_1,\tilde y_1,\varepsilon,A(\varepsilon)) =   o( \varepsilon ^r) ,
    \label{eq:second}
\end{equation} 
 thus implying that the left-hand side of \ref{eq:BB}  is $ o( \varepsilon ^r) $. 
\vskip 0.1 in 
\noindent {\bf  Proof of (\ref{eq:first}).}  By   (\ref{nRnS}) we have 
\begin{align*}
    &\quad _qS(\tilde x_1,\tilde y_1,\varepsilon,A(\varepsilon)) -\, _qS(x_0,y_0,\varepsilon,A(\varepsilon))\\
    &= -q A(\varepsilon) -\varepsilon \sum_{k=0}^{q-1} f(\tilde x_k) + q A(\varepsilon) + \varepsilon \sum_{k=1}^{q} f(\tilde x_k)\\
    &= -\varepsilon (f(\tilde x_q) - f(x_0))  . 
\end{align*}
But  
\begin{equation*}
        \tilde x_q \buildrel (\ref{remainder}) \over {=}  x_0 + 2p\pi + \, _qR\left(x_0,y_0, \varepsilon, A(\varepsilon)\right),
\end{equation*}
which together with (\ref{nRprincipal}) shows that the last difference in parentheses is $ O( \varepsilon^r)$, thus implying \ref{eq:first}. 
\vskip 0.1 in 
\noindent {\bf  Proof of (\ref{eq:second}).}
By Lemma \ref{lemma:remainders} 
\[
   _qS(\tilde x_1,\tilde y_1,\varepsilon,A(\varepsilon)) =  \sum_{m\ge1} \; _qS_m(\tilde x_1,\tilde y_1,\varepsilon,A(\varepsilon)) 
\]
and 
\[
   _qS(x_1,y_1,\varepsilon,A(\varepsilon)) = \sum_{m\ge1}\; _qS_m(x_1,y_1,\varepsilon,A(\varepsilon)).
\]
where $_qS_m(x,y,\varepsilon,A(\varepsilon))$ is a degree-$m$ homogeneous polynomial in the items from the list $\{  y-A(\varepsilon)-\varepsilon f(  x),\; -A(\varepsilon)-\varepsilon f( x + k\mu),\;  \varepsilon f^{(l)}(\  x+k\mu)\;\} $ with $  0\le k \le q-1,\ 1\le l \le m-1$.
We will show that the corresponding terms in each sum differ by $o(\varepsilon^r)$. We have 
\begin{align*}
    x_1 &= x_0 + y_0 + \mu - \Delta(x_0) - \varepsilon f(x_0)\\
    y_1 &= y_0 - \Delta(x_0) - \varepsilon f(x_0);
\end{align*}
and 
\begin{align*}
    \tilde x_1 &= x_0 + y_0 + \mu - A(\varepsilon) - \varepsilon f(x_0)\\
    \tilde y_1 &= y_0 - A(\varepsilon) - \varepsilon f(x_0).
\end{align*}
Since  $\Delta(x_0)=A(\varepsilon)+\Delta_r(x_0)\varepsilon^r+o(\varepsilon^r)$, this implies
\begin{equation}\label{eq:difference}
    \begin{split}
    \tilde x_1 -x_1 &= \Delta_r(x_0) \varepsilon^r + o(\varepsilon^r)\\
        \tilde y_1 -y_1 &= \Delta_r(x_0) \varepsilon^r + o(\varepsilon^r)\\
        f^{(l)}(\tilde x_1+k\mu) - f^{(l)}(x_1+k\mu) &= f^{(l+1)}( x_1+k\mu) \Delta_r(x_0) \varepsilon^r + o(\varepsilon^r).
    \end{split}
\end{equation}
This  $ O( \varepsilon^r) $ difference in $ x $ and $y $  results in the $ o( \varepsilon ^r) $ difference in the terms $ _qS_m $ as we now show.  Starting with $ m=1 $ we have   
\begin{equation*}
    \begin{split}
        &\quad _qS_1(\tilde x_1,\tilde y_1,\varepsilon,A(\varepsilon)) -\; _qS_1(x_1,y_1,\varepsilon,A(\varepsilon))\; \\
     &\buildrel (\ref{eq:firstorder}) \over {=}   - \varepsilon \sum_{k=0}^{q-1} (f(\tilde x_1+k\mu) - f(x_1+k\mu))
     = -\varepsilon \sum_{k=0}^{q-1} f'(x_1+k\mu) \Delta_r(x_0) \varepsilon^r + o(\varepsilon^r)= o(\varepsilon^r).
    \end{split}
\end{equation*}
 For  $ m\ge 2 $, according to   Lemma \ref{lemma:remainders},  each term in $_qS_m$  contains at least one derivative of $g(x;\varepsilon,A(\varepsilon))=-A(\varepsilon)-\varepsilon f(x)$ for both $_qS_m(x_1,y_1,\varepsilon,A(\varepsilon))$ and $_qS_m(\tilde x_1,\tilde y_1,\varepsilon,A(\varepsilon))$, which contributes a factor of $\varepsilon$. This, together with (\ref{eq:difference}) implies
\[
   _qS_m(\tilde x_1,\tilde y_1,\varepsilon,A(\varepsilon)) -\; _qS_m(x_1,y_1,\varepsilon,A(\varepsilon)) = o(\varepsilon^r). 
\]
 We showed that $ _qS(x_1, y_1, \varepsilon, A ( \varepsilon )) -\, _qS(\tilde x_1,\tilde  y_1, \varepsilon, A ( \varepsilon ))= o( \varepsilon^r) $. And since $ y_1=Y(x_1) $ this proves (\ref{eq:second}), thus completing the proof of the lemma. \end{proof}

\vskip 0.5 in 
 \section{End of proof of the main theorem}
In this last section we complete the proof of  Theorem \ref{thm:arnoldtongue} using the results of the previous sections. 
The main idea, similar to  \cite{Arnold1983}, is to observe that if $ f $ is a trigonometric polynomial of degree $d$ then $ \Delta_r $ is a trigonometric polynomial of degree $ rd$. And since $ \Delta_r $ is nonconstant (by the definition) and periodic of period $ 2 \pi p/q $, one must have $ rd>q $, so that 
$ r> [q/d] $. And this would complete the proof of the theorem, since $ \Delta (x, \varepsilon ) = A( \varepsilon) + \Delta_r(x) \varepsilon^r + o(\varepsilon^r) $ implies that the range of $ \delta $ for which $ p/q $-periodic points exist is at most   $ O( \varepsilon ^r ) $ with $ r> [q/d] $. 

It remains therefore to show  that $ \Delta_r $ is indeed a trigonometric polynomial of degree at most $ rd$. According to (\ref{nSprincipal})  
\[
    \varepsilon ^r \Delta_r(x) = q^{-1} \sum_{m\ge 1}\,_qS_m(x, Y(x, \varepsilon ), \varepsilon, A( \varepsilon ))+o( \varepsilon ^r),
 \] 
and we must show that the coefficient of  $ \varepsilon ^r $  in the above sum is a trigonometric polynomial  of degree at most $rd$. 
According to Lemma \ref{lemma:remainders},
$_qS_m(x,Y(x),\varepsilon,A(\varepsilon))$ is   homogeneous polynomial of degree $ m $ in the items from the list 
\begin{equation}
    \{Y(x)-A(\varepsilon)-\varepsilon f_0,\; -A(\varepsilon)-\varepsilon f_k,\; \varepsilon f^{(l)}_k\} 
    \label{eq:list}
\end{equation} 
 with $0\le k \le q-1, 1\le l \le m-1 $, and thus only finitely many  terms - namely the ones  with $ m\le r $  -  contribute to the coefficient of $ \varepsilon^r $. According to  Theorem \ref{thm:ycurve} the coefficients $ Y_n $ in the expansion  
\[
    Y(x) = Y_1(x)\varepsilon + Y_2(x)\varepsilon^2+\cdots 
\]
are $n^{\rm th}$ degree polynomials in $ f$, its shifts by $ \mu $ and its derivatives. The key point here is that the power of $ \varepsilon $ in $ Y(x)$ is also the degree of the   polynomial $ Y_k $ in $ f $, its derivatives and shifts. In addition,      $ \varepsilon $
enters with power $ 1 $ in the list (\ref{eq:list}) as a factor of every  $f$ and its derivatives and shifts (while $ A( \varepsilon ) $ has constant coefficients). This shows that the coefficient of 
$ \varepsilon^r $ is a polynomial of degree at most $ r $ in 
$f$, its derivatives and shifts. Moreover, this coefficient is a finite sum, since only finitely many terms ($m\le r$) contribute to it. Finally,   since $ f $ is a trigonometric polynomial of degree $d$ it follows that the coefficient of $ \varepsilon^r $ is a trigonometric polynomial of degree at most $ rd$, as claimed. 

This completes the proof of Theorem \ref{thm:arnoldtongue}.

\bibliography{travellingwave}

\end{document}